\RequirePackage[dvipsnames, usenames]{xcolor}
\documentclass[journal, twoside, web]{ieeecolor}

\usepackage[style=ieee, url=false, isbn=false]{biblatex}
\def\logoname{blank_logo}
\definecolor{subsectioncolor}{rgb}{0,0.541,0.855}

\usepackage{./TLC}

\usepackage{mathtools}
\usepackage{amsmath}
\let\labelindent\relax
\usepackage{enumitem}

\title{A homotopy theorem for incremental stability}

\author{Thomas Chaffey$^{1}$, Andrey Kharitenko$^{2}$, Fulvio Forni$^{1}$, Rodolphe
    Sepulchre$^{1,3}$
\thanks{$^{1}$University of Cambridge, Department of Engineering, Trumpington Street,
Cambridge CB2 1PZ.} 
\thanks{$^{2}$ Automatic Control Laboratory,
Swiss Federal Institute of Technology Zürich (ETHZ), 8092 Zürich, Switzerland.}
\thanks{$^{3}$KU Leuven, Department of
Electrical Engineering (STADIUS), KasteelPark Arenberg, 10, B-3001 Leuven, Belgium.}
\thanks{Emails: {\tt\small tlc37@cam.ac.uk},  {\tt\small akharitenko@student.ethz.ch}, {\tt
\small \{f.forni, r.sepulchre\}@eng.cam.ac.uk}.}
\thanks{The research leading to these results has received funding from the European Research Council under the
Advanced ERC Grant Agreement SpikyControl n.101054323.  The work of the first author
is supported by Pembroke College, University of Cambridge.}
}

\begin{document}
\maketitle

\begin{abstract}
    A theorem is proved to verify incremental stability of a feedback system via a homotopy from a known incrementally stable system.  
    A first corollary of that result is that incremental stability
    may be verified by separation of Scaled Relative Graphs, correcting two
    assumptions in \autocite[Theorem 2]{Chaffey2023}.  A second corollary provides an
    incremental version of the classical IQC stability theorem.
\end{abstract}

\section{Introduction}

There are two standard approaches to verifying absolute stability of a feedback
interconnection: the first, introduced by \textcite{Zames1966a}, is to work in an
extended space, which includes unstable signals, and show that all signals are, in
fact, stabilized by the feedback.  The second, pioneered by \textcite{Megretski1997},
is based upon a \emph{homotopy} argument: it is shown that a known
stable system can be perturbed continuously to produce the desired feedback
interconnection, in such a way that stability is never lost.  Such an argument has
two ingredients: firstly, one shows that a system remains stable
under a small perturbation, provided a gain bound is satisfied.  Secondly, one must
verify such a gain bound along the entire path of perturbations. In the
original theorems of Megretski and Rantzer \autocite{Megretski1997, Rantzer1997}, it was shown that
perturbations which were small in the gap metric \autocite{Georgiou1997} preserved
stability, and (soft) \emph{Integral Quadratic Constraints (IQCs)} were
used to verify gain bounds.

In this paper, we offer an alternative approach: an \emph{incremental} homotopy
argument is developed (Theorem \ref{thm:homotopy_incremental}), where gain bounds are
replaced by incremental gain bounds, and the gap metric is replaced by a minor
modification of the incremental small gain theorem
(Theorem~\ref{thm:inc_small_gain}).  In a similar vein to the
incremental versions of the small
gain theorem \autocite{Desoer1975}, the stronger assumption of incremental
gain allows well-posedness and causality assumptions to be weakened.  While in a
standard homotopy argument, well-posedness and causality of the feedback
interconnection must be assumed along the entire path, requiring reference to an
extended space, in the incremental setting,
only incremental boundedness of the operators must be assumed, which may be verified
without any extended space.  The setting of Theorem~\ref{thm:homotopy_incremental} is
more general than the typical IQC setting, as neither operator is assumed to be
linear nor time-invariant.

Theorem~\ref{thm:homotopy_incremental} is a standalone and improved version of an
incremental stability theorem recently proved by the authors in
\autocite{Chaffey2023}, in the context of a graphical stability criterion based on
the \emph{Scaled Relative Graph (SRG)}, a graphical representation of an operator
introduced by \textcite{Ryu2021}. In \autocite{Chaffey2023}, it was shown that the
SRG generalizes the Nyquist diagram of an LTI operator to an arbitrary nonlinear
operator, and a generalization of the Nyquist criterion was given in
\autocite[Theorem 2]{Chaffey2023}.   As a first corollary of
Theorem~\ref{thm:homotopy_incremental}, we reprove \autocite[Theorem 2]{Chaffey2023},
correcting two technical assumptions.  This theorem has also been recently
generalized by \textcite{Chen2024}, removing a technical assumption, however that
generalization relies on an extended space, which is not required here.  

As a second corollary to Theorem~\ref{thm:homotopy_incremental}, we obtain an incremental version of the IQC stability theorem of
\textcite{Megretski1997}.  This corollary is closely related to \autocite[Theorem 7.40]{Scherer2015}, but does not make any assumptions of causality, nor rely on an
extended space.
It has been shown that incremental gain bounds can
be verified using closely related \emph{differential} IQCs \autocite{Wang2019}.
Incremental IQCs are used in the study of periodic solutions by
\textcite{Jonsson2001}, in the study of neural networks by
\textcite{Gronqvist2022} and in system identification by \textcite{vanWaarde2023}. It has recently been
shown that well-posedness and causality assumptions can be relaxed in the
non-incremental setting \autocite{Freeman2022}.  

The remainder of this note is structured as follows.  In Section~\ref{sec:prelims},
we introduce necessary notation and preliminary results.  In
Section~\ref{sec:homotopy}, we prove our main result
(Theorem~\ref{thm:homotopy_incremental}), a general incremental homotopy theorem.  In
Section~\ref{sec:SRG}, it is shown how incremental stability may be verified using separation of
SRGs, and in Section~\ref{sec:IQC}, an incremental IQC stability theorem is given.
Finally, in Section~\ref{sec:relax}, it is shown how assumptions of incremental
boundedness may be relaxed in exchange for well-posedness and causality, giving a
middle ground between the incremental approach of
Theorem~\ref{thm:homotopy_incremental} and standard non-incremental arguments.

\section{Preliminaries}\label{sec:prelims}

\subsection{Signals and systems}

Let $\mathcal{F}$ denote the space of all functions mapping the interval $[0,
\infty)$ into $\R^n$.
Let $L_2^n$ be the space of equivalence classes of trajectories $u \in \mathcal{F}$ and satisfying
\begin{IEEEeqnarray*}{rCl}
    \norm{u} := \left(\int_0^\infty u(t)\tran u(t) \dd t\right)^{\frac{1}{2}} < \infty,
\end{IEEEeqnarray*}
under the equivalence $u \sim y \iff \norm{u - y} = 0$.  We will abuse terminology
in the usual way and say that a trajectory $u \in \mathcal{F}$ \emph{belongs to $L_2^n$} when
it belongs to an equivalence class in $L_2^n$.  For the remainder of this note, we
will drop the dimension $n$, and simply denote $L_2^n$ by $L_2$, where $n$ is
arbitrary.  Given an element $x \in L_2$, we let $\hat{x}$ denote its Fourier
transform.

By an \emph{operator} on  a domain $D \subseteq L_2$ we will mean a single-valued map
$H:D \to L_2$.  The domain $D$ will also be denoted $\dom H$.  We will associate an operator $H$ with its \emph{relation} or \emph{graph},
defined as $\{(u, y)\; | \; y = H(u)\} \subseteq L_2 \times L_2$, and denote the two in the same way.
Scalar multiplication, summation, and inversion of relations are defined as follows:
\begin{IEEEeqnarray*}{rCl}
    \alpha H &:=& \{(u, \alpha y)\; | \; y = H(u)\}\\
    H_1 + H_2 &:=& \{(u, y + z)\; | \; y = H_1(u), z = H_2(u)\}\\
    H^{-1} &:=& \{(y, u)\; | \; y = H(u)\}.
\end{IEEEeqnarray*}
We note that the inverse of an operator may be multivalued, and therefore not
necessarily an operator.   However, these relational operations are always well
defined.

Given an operator $H:L_2 \to L_2$, we define the \emph{gain} of $H$ on $L_2$, denoted $\norm{H}$, to be the smallest $\gamma > 0$ such
that there exists $\beta \in \R$ such that, for all inputs $u \in L_2$,
\begin{IEEEeqnarray*}{rCl}
    \norm{H(u)} \leq \gamma \norm {u} + \beta.
\end{IEEEeqnarray*}
If the gain of an operator is finite, the operator is said to be \emph{bounded}.  If
$\beta = 0$, the operator is said to have \emph{finite gain with zero offset}.

The \emph{incremental gain} of $H$ on $L_2$ is defined to be
\begin{IEEEeqnarray*}{rCl}
    \norm{H}_{\Delta} := \sup_{u_1, u_2 \in L_2, u_1 \neq u_2} \frac{\norm{H(u_1) - H(u_2)}}{\norm{u_1 - u_2}}.
\end{IEEEeqnarray*}
If the incremental gain of an operator is finite, the operator is said to be
\emph{incrementally bounded}.  If an operator derives from a dynamical system,
incremental boundedness is equivalent to asymptotic stability of any input/output
trajectory, under reachability and observability assumptions \autocite{Fromion1996}.

%


Consider the negative feedback interconnection of two operators, $H_1$ and $H_2$, defined by
the equations
\begin{IEEEeqnarray}{rCl}
    e &=& u - H_2(y)\label{eq:feedback_one}\\
    y &=& H_1(e)\label{eq:feedback_two}
\end{IEEEeqnarray}
and illustrated in Figure~\ref{fig:feedback}.
We make the standing assumption that 
this negative feedback interconnection defines a (single-valued) operator from some (possibly empty) domain $D
\subseteq L_2$ to $L_2$, mapping $u$ to $y$. We denote the relation of this operator
by $[H_1, H_2] := \{(u, y) \; | ;\ \text{there exits a unique } e \text{ s.t. }
\eqref{eq:feedback_one}--\eqref{eq:feedback_two} \text{ are satisfied} \}$.  

\begin{figure}[h]
    \centering
    \includegraphics{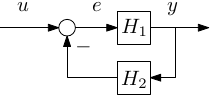}
    \caption{Negative feedback interconnection of $H_1$ and $H_2$.}
    \label{fig:feedback}
\end{figure}

We will make use of the following two technical lemmas.

\begin{lemma}
    Given operators $H_1, H_2 : L_2 \to L_2$, 
    \begin{IEEEeqnarray*}{rCl}
        [H_1, H_2] = (H_1^{-1} + H_2)^{-1}.
    \end{IEEEeqnarray*}
\end{lemma}

\begin{proof}
    Applying the definitions of relational inverse and sum, we arrive at
    \begin{IEEEeqnarray*}{rCl}
        (H_1^{-1} + H_2)^{-1} &=& \{ (e + z, y) \; | \; y = H_1(e), u - e = H_2(y)\}.
    \end{IEEEeqnarray*}
Setting $u = e+z$, we arrive at the definition of $[H_1, H_2]$.
\end{proof}

\begin{lemma}\label{lem:feedback_identity}
Given $\tau, \nu \geq 0$ and operators $H_1, H_2: L_2 \to L_2$, 
\begin{IEEEeqnarray}{rCl}
[H_1, (\tau + \nu)H_2] = [[H_1, \tau H_2], \nu H_2].
\end{IEEEeqnarray}
\end{lemma}


\begin{proof}
Note that $(H^{-1})^{-1} = H$, and, given a single-valued operator, $(\tau + \nu)H = \tau H + \nu H$.  We then have:
    \begin{IEEEeqnarray*}{+rCl+x*}
    [H_1, (\tau + \nu)H_2] &=& (H_1^{-1} + (\tau + \nu)H_2)^{-1}\\
                           &=& (H_1^{-1} + \tau H_2 + \nu H_2)^{-1}\\
                           &=& ((H_1^{-1} + \tau H_2)^{-1})^{-1} + \nu H_2)^{-1}\\
                           &=& [[H_1, \tau H_2], \nu H_2]. & \qedhere
    \end{IEEEeqnarray*}
\end{proof}

The following theorem is a modified version of the classical incremental small gain
theorem \autocite[Theorem 30, p. 184]{Desoer1975}.  It differs from the classical statement of the
theorem in that the operators are defined only on $L_2$, rather than an extended
space, and are not required to map $0$ to $0$. It is closely related to the incremental gap robustness result of \textcite[Theorem 1]{Georgiou1997}.

\begin{theorem}[Incremental small gain theorem]\label{thm:inc_small_gain}
    Let $H_1, H_2: L_2 \to L_2$ be operators with incremental gain bounds of
    $\gamma_1, \gamma_2$ respectively.  If $\gamma_1\gamma_2 < 1$, then for any 
    $u \in L_2$, there exist unique $e, y \in L_2$ satisfying the feedback
    interconnection \eqref{eq:feedback_one}--\eqref{eq:feedback_two}.
\end{theorem}

\begin{proof}
    Fix $u \in L_2$.  Substituting \eqref{eq:feedback_two} in \eqref{eq:feedback_one}
    gives $e = u - H_2(H_1(e))$. Define
    $K_{u}(x) := u - H_2(H_1(x))$.  Since $H_1, H_2:L_2 \to L_2$, $K_{u}:
     L_2 \to L_2$. We claim that $K_{u}$ is a contraction on $L_2$.  Indeed, letting $x, \bar{x} \in L_2$, we have
    \begin{IEEEeqnarray*}{rCl}
        \norm{K_{u}(x) - K_{u}(\bar{x})} &=&  \norm{u - H_2(H_1(x)) - u +
        H_2(H_1(\bar{x}))}\\
        &=& \norm{H_2(H_1(\bar{x}) - H_2 (H_1(x))}\\
&\leq& \gamma_1 \norm{H_2(\bar{x}) - H_2(x)}\\
&\leq& \gamma_1\gamma_2 \norm{x - \bar{x}}.
    \end{IEEEeqnarray*}
Therefore, by the Banach fixed point theorem, there exists a unique solution to $e = K_{u}(e)$  for
each $u \in L_2$.  Furthermore, we have $y = H_1(e)$, so existence and uniqueness of
$y$ is guaranteed.
\end{proof}

\subsection{Scaled Relative Graphs}

The Scaled Relative Graph (SRG) is a graphical representation of the gain and phase
of an operator.  Phase is given by the angle between two signals.  For $u, y \in
L_2$, this  is defined as
\begin{IEEEeqnarray*}{rCl}
        \angle(u, y) := \arccos \frac{\Re \ip{u}{y}}{\norm{u}\norm{y}} \in [0, \pi]. 
\end{IEEEeqnarray*}

We define the SRG for an arbitrary relation, allowing us to talk about the SRG of an
operator and its relational inverse on the same footing.
Let $R \subseteq L_2 \times L_2$.  We write $u \in R(y)$ if $(u, y) \in R$.  Given $u_1, u_2 \in
L_2$, $u_1 \neq u_2$, define the set of complex numbers $z_R(u_1, u_2)$ by
\begin{IEEEeqnarray*}{rCl}
        z_R(u_1, u_2) \coloneqq &&\left\{\frac{\norm{y_1 - y_2}}{\norm{u_1 - u_2}} e^{\pm j\angle(u_1 -
u_2, y_1 - y_2)}\right.\\&&\bigg|\; y_1 \in R(u_1), y_2 \in R(u_2) \bigg\}.
\end{IEEEeqnarray*}
If $u_1 = u_2$ and there are corresponding
outputs $y_1 \neq y_2$, then
$z_R(u_1, u_2)$ is defined to be $\{\infty\}$.  If $R$ is single valued at $u_1$,
$z_R(u_1, u_1)$ is the empty set.
The \emph{Scaled Relative Graph} (SRG) of $R$ is then given by
\begin{IEEEeqnarray*}{rCl}
        \srg{R} := \bigcup_{u_1, u_2 \in\, L_2}  z_R(u_1, u_2).
\end{IEEEeqnarray*}
The SRG of an operator $H: L_2 \to L_2$ is defined to be the SRG of its relation.
We refer the reader to \autocite{Chaffey2023} for the relationship between the SRG,
the Nyquist diagram of a transfer function and the incremental disc of a static
nonlinearity, and to \autocite{Pates2021} for the relationship to the numerical range
of a linear operator.

The incremental gain of an operator (on its domain) is the maximum radius of its SRG.  The proof is
immediate from the definition of the SRG.

\begin{lemma}\label{lem:gain_bound}
    Given an operator $H: D \to L_2$, 
\begin{IEEEeqnarray*}{rCl}
\sup_{u_1, u_2 \in D} \frac{\norm{H(u_1) - H(u_2)}}{\norm{u_1 - u_2}} = \sup_{z \in \srg{H}}
    |z|.
\end{IEEEeqnarray*}
\end{lemma}

Given an interconnection of operators, the SRG of the interconnection can be
estimated from the SRGs of the individual operators.  We recall a couple of necessary
results in the following lemma, and refer the reader to \autocite{Ryu2021} for the full theory.

An SRG $\mathcal{G}$ is said to satisfy the \emph{chord property} if, for each $z \in
\mathcal{G}$, $\lambda z + (1-\lambda)\bar{z} \in \mathcal{G}$ for all $\lambda \in
[0, 1]$.
\begin{lemma}\label{lem:interconnections}
    Given $A, B \subseteq L_2 \times L_2$, we have:
    \begin{enumerate}
        \item $\srg{A^{-1}} = \{ z^{-1} \; | \; z \in
    \srg{A} \}$;\label{lem:inversion}.
    \item if $\bar{A} \supseteq \srg{A}$ is any set satisfying the chord property, then 
    $\srg{A + B} \subseteq \bar{A} + \srg{B}$.\label{lem:summation}
    \end{enumerate}
\end{lemma}

\begin{proof}
    The property \ref{lem:inversion} is proved in \autocite[Theorem 5]{Ryu2021}.  The
    proof of Property \ref{lem:summation} proceeds as follows. Let $(u_1, y_{A}),
    (u_2,z_A) \in A, (u_1, y_B), (u_2, z_B) \in B$.   Then $(u_1, y_A +
    y_B), (u_2, z_A + z_B) \in A + B$, and
    \begin{IEEEeqnarray*}{lCl}
        w = \frac{\norm{y_A + y_B - z_A - z_B}}{\norm{u_1 - u_2}}\\
\exp(j \angle(y_A + y_B - z_A - z_B, u_1 - u_2)) \in \srg{A + B},\\
w_A = \frac{\norm{y_A - z_A}}{\norm{u_1 - u_2}} \exp(j \angle{y_A - z_A, u_1 - u_2}) \in \srg{A},\\
w_B = \frac{\norm{y_B - z_B}}{\norm{u_1 - u_2}} \exp(j \angle{y_B - z_B, u_1 - u_2}) \in \srg{B}.\\
    \end{IEEEeqnarray*}
Then, by direct calculation, $\Re(w) = \Re(w_A) + \Re(w_B)$ and $\Im(w_B) - \Im(w_A)
\leq \Im(w) \leq \Im(w_B) + \Im(w_A)$, that is, $w \in w_B + [w_A, \overline{w_A}]$.
Since $[w_A, \overline{w_A}] \subseteq \bar{A}$, the claim follows.
\end{proof}
\section{Incremental homotopy}\label{sec:homotopy}

The following theorem gives a method for verifying finite incremental gain of a
feedback interconnection using a homotopy from a known incrementally bounded
operator.    The proof proceeds by
using the incremental small gain theorem to show small perturbations in the feedback
preserve stability, and applies this idea inductively to scale the feedback from $0$
to $1$. 

\begin{theorem}[Incremental homotopy]\label{thm:homotopy_incremental}
    Let $H_1, H_2: L_2 \to L_2$ be operators such that
    \begin{enumerate}[label=(\roman*)]
        \item  $H_1, H_2$  have finite incremental gain;
        \item  there exists $\gamma >0$ such that, for all $\tau \in [0, 1]$ and all
        $u_1, u_2 \in \dom{[H_1, \tau H_2]}$, we have
        \begin{IEEEeqnarray*}{rCl}
            \norm{y_1 - y_2} \leq \gamma \norm{u_1 - u_2},
        \end{IEEEeqnarray*}
        where $y_i = [H_1, \tau H_2](u_i), i=1,2$.\label{cond:bound}
    \end{enumerate} 
    Then $\dom{[H_1, H_2]} = L_2$ and $[H_1, H_2]$ has an incremental
    gain bound of $\gamma$.
\end{theorem}

\begin{proof}
    Let $\gamma_1, \gamma_2$ be incremental gain bounds for $H_1$ and $H_2$
    respectively.  We begin by showing that there exists $\nu > 0$ such that $\dom{[H_1,
    \nu H_2]} = L_2$.  Indeed, setting $\nu < 1/(\gamma_1\gamma_2)$, we
    have that $\nu \gamma_1 \gamma_2 < 1$, so it follows from
    Theorem~\ref{thm:inc_small_gain} that $\dom{[H_1, \nu H_2]} = L_2$.

    By assumption, $\gamma$ is an incremental gain bound
    for $[H_1, \nu H_2]$.  It then follows from Theorem~\ref{thm:inc_small_gain} that,
    for all $\tau \in [0, 1/\gamma\gamma_2)$, $\dom{[[H_1, \nu H_2], \tau H_2]} = L_2$.  From Lemma~\ref{lem:feedback_identity}, $[[H_1, \nu H_2], \tau H_2] =
    [H_1, (\nu + \tau) H_2]$.  Again, by assumption, this operator has an incremental
    gain bound of $\gamma$.

    Proceeding inductively, we have that $\dom{[H_1, (\nu + k\tau) H_2]} = L_2$
    for all $\tau \in [0, 1/\gamma\gamma_2)$ and positive integers $k$ such that
    $\nu + k\tau \leq 1$, so, in particular, $\dom{[H_1, H_2]} = L_2$.  The
    incremental gain bound of $\gamma$ then follows from Condition~(\ref{cond:bound})
    in the theorem statement.
\end{proof}

We note that \eqref{cond:bound} implies single-valuedness of $[H_1, \tau H_2]$ for
all $\tau \in [0, 1]$.

Theorem~\ref{thm:homotopy_incremental} allows incremental stability of a feedback
interconnection to be verified using only information about $L_2$ signals, without
any reference to an extended space containing unbounded signals.  This is useful
as it allows the use of operator theoretic tools to verify the incremental gain bound of
condition~\ref{cond:bound}.  In the following section, we show this may be done
graphically using the SRG, and in Section~\ref{sec:IQC} we verify the incremental
gain bound using IQCs.

\section{SRG separation}\label{sec:SRG}

Given an SRG $\mathcal{G}$, let $\overline{\mathcal{G}}$ denote the smallest SRG
containing $\mathcal{G}$ and satisfying the chord property.  We begin with a
graphical condition which guarantees finite incremental gain of a feedback
interconnection, on its domain (which may not, in general, be all of $L_2$).
Given two regions $X_1, X_2$ in the extended complex plane, we let
$\operatorname{dist}(X_1, X_2)$ denote $\inf_{x_1 \in X_1, x_2 \in X_2} |x_1 - x_2|$.

\begin{lemma}\label{lem:separation_bound}
    Let $H_1, H_2$ be operators, and suppose there exists $\gamma > 0$ such that
    $\operatorname{dist}(\srg{H_1}^{-1}, -\overline{\srg{H_2}}) \geq 1/\gamma$. Then, for
    any $u_i \in \dom{[H_1, H_2]}$ and $y_i = (H_1^{-1} + \tau H_2)^{-1}(u_i)$, we have 
    $\norm{y_1 - y_2} \leq \gamma \norm{u_1 - u_2}$.
\end{lemma}

\begin{proof}
    Since $\overline{\srg{H_2}^{-1}}$ satisfies the chord property by construction,
    we can apply the SRG sum rule.  The result then follows from
    Lemmas~\ref{lem:gain_bound} and~\ref{lem:interconnections} and the following
    geometry.

    \begin{center}
    \includegraphics[width=\linewidth]{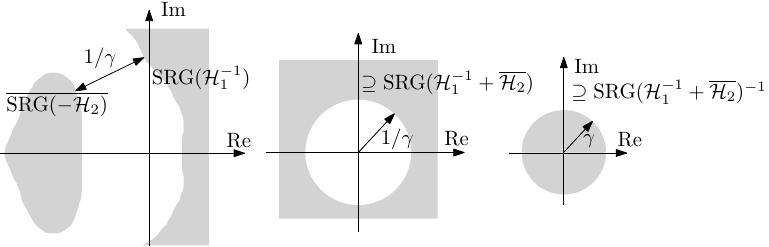}
    \end{center}
\end{proof}

The separation condition of Lemma~\ref{lem:separation_bound} guarantees that an
operator has an incremental bound on its domain.  Guaranteeing that this domain is,
in fact, $L_2$, requires a stronger separation property: the SRGs
must remain separated as the feedback is gradually increased from zero.  We formalize
this separation property as follows.

\begin{definition}
    Two operators $H_1$ and $H_2$ are said to have \emph{strictly separated SRGs with
    margin $r_{\min} > 0$} if
    $\operatorname{dist}(\srg{H_1}^{-1}, -\tau\overline{\srg{H_2}}) \geq r_{\min}$ for all $\tau
    \in (0, 1]$.
\end{definition}

\begin{corollary}\label{cor:SRG_stability}
    Suppose $H_1, H_2: L_2 \to L_2$  have bounded incremental gain and strictly separated SRGs. Then the feedback
    interconnection $[H_1, H_2]$ maps $L_2$ to $L_2$ and has bounded incremental
    gain.
\end{corollary}

\begin{proof}
    It follows from strict separation of the SRGs of $H_1$ and $H_2$, and
    Lemma~\ref{lem:separation_bound}, that $1/r_{\min}$ is an incremental gain bound
    for $[H_1, \nu H_2]$.  The result then follows from
    Theorem~\ref{thm:homotopy_incremental}.
\end{proof}

Corollary~\ref{cor:SRG_stability} corrects the assumptions of \autocite[Theorem
2, Corollary 1, Corollary 2]{Chaffey2023} in two ways: in
\autocite{Chaffey2023}, only $H_1$ was was assumed to be incrementally bounded, and
the separation of the SRGs was not required to be strict.
Strict separation must also be assumed in \autocite[Theorem 1]{Chaffey2023}.  The additional assumptions are
satisfied in all the examples of \autocite{Chaffey2023}.  Furthermore, the recent
result of \autocite[Theorem 1]{Chen2024} removes these two assumptions, but
at the price of a weaker form of incremental stability.
The following example illustrates why strict separation of the SRGs
is required to guarantee incremental boundedness.

\begin{example}
    Consider the operator $N:L_2 \to L_2$ given by
    \begin{IEEEeqnarray*}{rCl}
    (Nu)(t) = \phi(u(t))
    \end{IEEEeqnarray*}
    where $\phi(x) = -\arctan(x)$.
    $\phi$ is $1$-Lipschitz continuous, so $N$ has an incremental gain bound of
    $1$.  It follows that its SRG is contained in the closed unit disc.  We now
    observe the following: the SRG of $N$ is, in fact, contained in the open unit
    disc $\mathbb{D} := \{z \in \C\; | \; |z| < 1\}$, and therefore does not contain 
    the point $-1$.  However, the operator
    formed by putting $N$ in unity gain negative feedback has infinite incremental
    gain.

    To see that $\srg{N} \subseteq \mathbb{D}$, suppose that there exist $u_1, u_2 \in
    L_2$, $u_1 \neq u_2$, such that
    \begin{IEEEeqnarray}{rCl}
    \frac{\norm{N(u_1) - N(u_2)}}{\norm{u_1 - u_2}} = 1. \label{eq:example_1}
    \end{IEEEeqnarray}
    Since $u_1 \neq u_2$, the set $E = \{t \in \R \; | \; u_1(t) \neq u_2(t) \}$ has positive
    measure.  Equation~\eqref{eq:example_1} implies
    \begin{IEEEeqnarray*}{rCl}
    \int_E |u_1(t) - u_2(t)|^2 \dd{t} = \int_E |\phi(u_1(t)) - \phi(u_2(t))|^2 \dd{t}\\
                                  = \int_E \left|\frac{\phi(u_1(t)) -
                              \phi(u_2(t))}{u_1(t) - u_2(t)}\right|^2 |u_1(t) - u_2(t)|^2
                              \dd{t},
    \end{IEEEeqnarray*}
    or equivalently
    \begin{IEEEeqnarray*}{rCl}
        \int_E \left(1 - \left|\frac{\phi(u_1(t)) -
                              \phi(u_2(t))}{u_1(t) - u_2(t)}\right|^2\right)|u_1(t) - u_2(t)|^2
                              \dd{t} = 0.
    \end{IEEEeqnarray*}
    Hence we must have
    \begin{IEEEeqnarray*}{rCl}
        \left|\frac{\phi(u_1(t)) - \phi(u_2(t))}{u_1(t) - u_2(t)}\right| = 1
    \end{IEEEeqnarray*}
    for almost all $t \in E$.  But this is impossible, since 
    \begin{IEEEeqnarray*}{rCl}
\left|\frac{\phi(x) - \phi(y)}{x-y}\right| < 1
    \end{IEEEeqnarray*}
    for any $x, y \in \R$ with $x \neq y$.  Thus indeed $\srg{N} \subseteq \mathbb{D}$.

    We now show the $(N^{-1} + I)^{-1}$ does not have a finite incremental gain.  Indeed, on its
    domain, we have
    \begin{IEEEeqnarray*}{rCl}
    (N^{-1} + I)^{-1}(u)(t) = \psi^{-1}(u(t)),
    \end{IEEEeqnarray*}
    where $\psi(x) = x - \tan(x)$.  However, the function $\psi^{-1}$ is not
    Lipschitz continuous at $x = 0$, since $\psi^\prime(0) = 0$ so
    $(\psi^{-1})^\prime(0)$ does not exist.  The ratio
    \begin{IEEEeqnarray*}{rCl}
    \frac{\norm{(N^{-1} + I)^{-1}(u_1) - (N^{-1} + I)^{-1}(u_2)}}{\norm{u_1 - u_2}}
    \end{IEEEeqnarray*}
    can be made arbitrarily large, for example by taking $u_1(t) =
    au(t)$, $u_2(t) = u(t)$ for $a, b \neq 0$ small
    enough, where 
    \begin{equation*}
\begin{gathered}[b]
        u(t) := \begin{cases}
            1 & t \in [0, 1]\\
            0 & \text{otherwise.}
        \end{cases}
\\
\end{gathered}\qedhere
    \end{equation*}
\end{example}

\section{Incremental IQCs}\label{sec:IQC}

In this section, we give a second method of verifying condition~\ref{cond:bound} of
Theorem~\ref{thm:homotopy_incremental}: satisfaction of an incremental IQC.  This
gives a incremental version of the classical IQC stability theorem 
\autocite[Theorem 1]{Megretski1997}.  Our theorem is closely related to \autocite[Theorem
7.40]{Scherer2015}, but does not rely on an extended space or assumptions of
causality.  The developments closely follow the non-incremental theory of
\autocite{Rantzer1997, Freeman2022}.

\begin{corollary}\label{cor:IQC}
    Let $H_1: L_2 \to L_2$ be a bounded LTI operator, and $H_2: L_2 \to L_2$ be
    incrementally bounded.  Let $\Pi: j\R \to \C^{n \times n}$ be a Hermitian-valued
    function with $L_\infty$ entries.  Given $y_1, y_2 \in L_2$, define $\Delta
    \hat{y}(j\omega) := \hat{y}_1(j\omega) - \hat{y}_2(j\omega)$ and $\Delta
    \hat{H}_2(y)(j\omega) := \hat{H}_2(y_1)(j\omega) -
    \tau \hat{H}_2(y_2)(j\omega)$. Suppose that, for every
    $\tau \in [0, 1]$ and every $y_1, y_2 \in L_2$, we have
    \begin{IEEEeqnarray}{rCl}
{\footnotesize
    \int_{-\infty}^\infty \begin{pmatrix}
        \Delta \hat{y}(j\omega) \\ \tau \Delta
    \hat{H}_2(y)(j\omega) \end{pmatrix}^*  \Pi(j\omega) \begin{pmatrix}
        \Delta \hat{y}(j\omega) \\ \tau \Delta
    \hat{H}_2(y)(j\omega)  \end{pmatrix} \dd{\omega} \geq 0,
    \label{eq:inc_IQC_1}}
    \end{IEEEeqnarray}
    and that there exists $\epsilon > 0$ such that
    \begin{IEEEeqnarray}{rCl}
        \begin{pmatrix}
            \hat{H_1}(j\omega) \\ I
        \end{pmatrix}^* \Pi(j\omega) \begin{pmatrix}
            \hat{H_1}(j\omega) \\ I
        \end{pmatrix} \leq -\epsilon I \label{eq:inc_IQC_2}
    \end{IEEEeqnarray}
    for all $\omega \in \R$.  Then $[H_1, H_2]$ maps $L_2$ to $L_2$ and has bounded
    incremental gain.
\end{corollary}

In contrast to \autocite[Theorem 1]{Megretski1997}, Corollary~\ref{cor:IQC} does not require any well-posedness assumptions along
the homotopy path.
Before giving the proof of Corollary~\ref{cor:IQC}, we make a technical definition
and prove two lemmas, on which the proof relies.

\begin{definition}
    A functional $\sigma: L_2 \to \R$ is said to be \emph{quadratically continuous}
    if, for every $\epsilon > 0$, there exists $C > 0$ such that
    \begin{IEEEeqnarray*}{rCl}
    \sigma(y) \leq \sigma(x) + \epsilon \norm{x}^2 + C\norm{x - y}^2
    \end{IEEEeqnarray*}
    for all $x, y \in L_2$.
\end{definition}

The class of quadratically continuous functionals that we will use are characterized
in the following lemma, which is a special case of \autocite[Lem. 3.18]{Freeman2022}.

\begin{lemma}\label{lem:quadratic_continuous}
    Let $\Pi: j\R \to \C^{n \times n}$ be a Hermitian-valued function with $L_\infty$
    entries.  Then the functional $\sigma:L_2 \to \R$ defined by 
    \begin{IEEEeqnarray*}{rCl}
    \sigma(x) &:=& \ip{x}{x}_{\Pi}\\
    \ip{x}{y}_{\Pi} &:=& \Re \int_{-\infty}^\infty \hat{x}^*(j\omega) \Pi(j\omega)
    \hat{y}(j\omega)
    \dd{\omega}
    \end{IEEEeqnarray*}
    is quadratically continuous.
\end{lemma}

\begin{proof}
    Note that $\ip{\cdot}{\cdot}_{\Pi}$ defines a bounded Hermitian form.
   \begin{IEEEeqnarray*}{rCl}
   \sigma(x) - \sigma(y) &=& \ip{y}{y}_{\Pi} - \ip{x}{x}_{\Pi}\\
                         &=& \ip{y}{y}_{\Pi} - \ip{x}{y}_{\Pi} + \ip{x}{y}_{\Pi} - \ip{x}{x}_{\Pi}\\
                         &=& \ip{y-x}{y}_{\Pi} - \ip{x}{y-x}_{\Pi}\\
                         &=& \ip{y-x}{y}_{\Pi} - \ip{y-x}{x}_{\Pi}\\&& +
                         \ip{y-x}{y}_{\Pi} + \ip{x}{y-x}_{\Pi}\\
                         &=& \ip{y-x}{y - x}_{\Pi} + \ip{y-x}{x}_{\Pi} +
                         \ip{x}{y-x}_{\Pi}.
   \end{IEEEeqnarray*} 
   Since $\ip{\cdot}{\cdot}_{\Pi}$ is bounded, there exists $M \geq 0$ such that $\ip{x}{y}
   \leq M \norm{x}\norm{y}$ for all $x, y \in L_2$.  We therefore have
   \begin{IEEEeqnarray}{rCl}
   \sigma(x) - \sigma(y) &\leq& M\norm{x-y}^2 +
   2M\norm{x}\norm{x-y}.\label{eq:quad_1}
   \end{IEEEeqnarray}
   Furthermore, for any $\epsilon > 0$,
   \begin{IEEEeqnarray*}{lCl}
        \frac{1}{\epsilon} \left( M \norm{x-y} - \epsilon \norm{x}\right)^2 =\\
            \frac{1}{\epsilon} M^2 \norm{x-y}^2 - 2M\norm{x}\norm{x-y} + \epsilon
            \norm{x}^2 \geq 0\\
            \text{so}\quad 2M \norm{x}\norm{x-y} \geq \frac{1}{\epsilon} M^2
            \norm{x-y}^2 + \epsilon\norm{x}^2.
   \end{IEEEeqnarray*}
    Combining with \eqref{eq:quad_1}, we have
    \begin{IEEEeqnarray*}{+rCl+x*}
    \sigma{x} - \sigma{y} &\leq&\left(M + \frac{1}{\epsilon}M^2\right)\norm{x-y}^2 +
    \epsilon\norm{x}^2. & \qedhere
    \end{IEEEeqnarray*}
\end{proof}

We now show that a quadratically continuous functional can be used to verify incremental
boundedness.

\begin{lemma}\label{lem:IQC_bound}
    Consider $H_1, H_2: L_2 \to L_2$ and assume $H_1$ has an incremental gain bound
    of $\lambda$. Let $\sigma: L_2 \to
    \R$ be quadratically continuous with constant $C$.  For $u_i \in L_2$, let $y_i
    \in [H_1,
    H_2](u_i)$ and $e_i$ denote $u_i - H_2(y_i)$.
    Suppose that, for all $u_1, u_2 \in \dom{[H_1, H_2]}$,
    we have
    \begin{IEEEeqnarray}{rCl}
    \sigma(h_1) \leq -2 \epsilon \norm{h_1}^2\label{eq:IQC_bound_1}\\
    \sigma(h_2) \geq 0, \label{eq:IQC_bound_2}
    \end{IEEEeqnarray}
    where $h_1 := (y_1 - y_2, e_1 - e_2)$, $h_2 := (y_1 - y_2, H_2(y_1) -
    H_2(y_2))$.  Then there exists $\lambda > 0$ such that, for
    all $u_1, u_2$, we have
    \begin{IEEEeqnarray*}{rCl}
    \norm{y_1 - y_2} \leq \sqrt{\frac{C}{\epsilon}\left(\frac{\lambda^2}{1 +
        \lambda^2}\right)} \norm{u_1 - u_2}.
    \end{IEEEeqnarray*}
\end{lemma}

\begin{proof}
    Given two signals $x_1, x_2$, let $\Delta x := x_1 - x_2$.
    We have
    \begin{IEEEeqnarray*}{rCl}
        0 &\leq& \sigma(h_2)\\
          &\leq& \sigma(h_1) + \epsilon \norm{h_1}^2 + C\norm{h_2 - h_1}^2\\
          &\leq& - \epsilon \norm{h_1}^2 + C\norm{h_2 - h_1}^2\\
          &=& -\epsilon(\norm{\Delta y}^2 + \norm{\Delta e}^2) + C(\norm{\Delta H_2(y) - \Delta
          e}^2)\\
          &=& -\epsilon(\norm{\Delta y}^2 + \norm{\Delta e}^2) + C\norm{\Delta u}^2\\
          &\leq& -\epsilon\left(1 + \frac{1}{\lambda^2}\right)\norm{\Delta y} +
          C\norm{\Delta u}^2,
    \end{IEEEeqnarray*}
    from which the result follows.
\end{proof}
    
\begin{proof}[Proof of Corollary~\ref{cor:IQC}]
    Define $\ip{\cdot}{\cdot}_{\Pi}$ as in Lemma~\ref{lem:quadratic_continuous}.  It
    follows from that lemma that $\sigma(x) := \ip{x}{x}_{\Pi}$ is quadratically
    continuous.  Equation~\eqref{eq:inc_IQC_1} gives condition~\eqref{eq:IQC_bound_2}
    of Lemma~\ref{lem:IQC_bound}.  We now show that Equation~\eqref{eq:inc_IQC_2}
    gives condition~\eqref{eq:IQC_bound_1}.  Indeed, pre- and post-multiplying with
    $\hat u^*(j\omega)$ and $\hat u(j\omega)$, respectively, and integrating over $\omega$, gives
    \begin{IEEEeqnarray*}{rCl}
    {\small
    \int_{-\infty}^\infty \begin{pmatrix}
        \hat{H_1}(j\omega)\hat{u}(j \omega) \\ \hat u(j \omega) \end{pmatrix}^*  \Pi(\omega)\begin{pmatrix}
        \hat{H_1}(j\omega)\hat{u}(j \omega) \\ \hat u(j \omega) \end{pmatrix}  \dd{\omega}
    \leq - \epsilon \norm{\hat u}^2.}
    \end{IEEEeqnarray*}
    Now let $\alpha > 0$ be a gain bound for $H_1$.  Then we can write $-\epsilon
    \norm{u}^2 \leq -2\bar{\epsilon}(\norm{u}^2 + \norm{H_1(u)}^2)$ for some
    $\bar{\epsilon} \leq \epsilon/(2(1 + \alpha))$.  Since $H_1$ is LTI, this
    non-incremental condition is then equivalent to the incremental condition
    \eqref{eq:IQC_bound_1}.

    We finally note that the gain bound given by Lemma~\ref{lem:IQC_bound} depends
    only on $\sigma$ and the incremental gain of $H_1$, and not on $\tau$.  The
    conclusions of the corollary then follow from
    Theorem~\ref{thm:homotopy_incremental}.
\end{proof}

\section{Relaxing the assumption of incremental boundedness}\label{sec:relax}
Much of the existing stability literature focuses on the verification of
non-incremental stability, and assumes weaker non-incremental boundedness properties
for the components.  In exchange, systems must be assumed to be well-posed, in the
sense of causality and existence of solutions in an appropriate ambient space,
containing $L_2$ but allowing unbounded signals.  In this section, we show that incremental stability can also be
verified under these weaker assumptions, subject to the same well-posedness 
assumptions as in a typical non-incremental analysis.  We begin by introducing the extended $L_2$ space.

Given $T > 0$, denote by $P_T: \mathcal{F} \to \mathcal{F}$ the truncation operator 
\begin{IEEEeqnarray*}{rCl}
    P_T(u)(t) := \begin{cases} u(t) & t < T\\
                                0 & t \geq T.
                \end{cases}
\end{IEEEeqnarray*}
The \emph{extended $L_2$ space}, $L_{2e}$, is defined as the subset of
$\mathcal{F}$ such that $P_T u \in L_2$ for all $T$.  An operator $H: L_2 \to L_2$,
or $H_e :L_{2e} \to L_{2e}$, is said to be \emph{causal} if $P_T H P_T = P_T H$
for all $T>0$.

A negative feedback interconnection is said to be
\emph{well-posed} if, for any $u \in L_{2e}$, there
exist unique $e, y \in L_{2e}$ satisfying \eqref{eq:feedback_one}--\eqref{eq:feedback_two}. 

\begin{theorem}\label{thm:homotopy_non_incremental}
    Suppose 
    \begin{enumerate}[label=(\roman*)]
        \item  $H_1, H_2: L_2 \to L_2$  have finite gain with zero offset 
    and are causal;
        \item $[H_1, \tau H_2]$ 
    is well-posed and causal for all $\tau \in (0, 1]$;
        \item  there exists $\gamma >0$ such that, for all $\tau \in [0, 1]$ and all
        $u_1, u_2 \in \dom{[H_1, \tau H_2]}$, we have
    \begin{IEEEeqnarray*}{rCl}
     \norm{y_1 - y_2} \leq \gamma \norm{u_1 - u_2},
    \end{IEEEeqnarray*}
    where $y_i = [H_1, \tau H_2](u_i), i = 1,2$.\label{cond:bound2}
        \end{enumerate}
    Then $[H_1, H_2]$ maps
    $L_2$ to $L_2$ and has finite incremental gain.
\end{theorem}

This theorem provides a middle ground between the incremental
Theorem~\ref{thm:homotopy_incremental} and classical homotopy results such as
\autocite[Theorem 1]{Megretski1997}.

\begin{proof}[Proof of Theorem~\ref{thm:homotopy_non_incremental}]
    The proof mirrors that of Theorem~\ref{thm:homotopy_incremental}, but replacing
    the incremental small gain theorem with its non-incremental version -- see, for
    example, \autocite[Theorem 1, p. 41]{Desoer1975}, with the modified condition
    suggested in Equation (8c) on the same page.
\end{proof}

As in the case of Theorem~\ref{thm:homotopy_incremental},
condition~\ref{cond:bound2} of Theorem~\ref{thm:homotopy_non_incremental} can be verified using SRG separation or incremental
IQCs.  In the case of SRG separation, we have the following result, the proof of
which is similar to Corollary~\ref{cor:SRG_stability}.    The
corresponding result for IQCs is similary to Corollary~\ref{cor:IQC} but incorporates
the assumptions of Theorem~\ref{thm:homotopy_non_incremental}.

\begin{corollary}\label{cor:SRG_stability2}
    Suppose $H_1, H_2: L_2 \to L_2$  have finite gain with zero offset. Suppose
    that $H_1$ and $H_2$ have strictly separated SRGs. Then the feedback
    interconnection $[H_1, H_2]$ maps $L_2$ to $L_2$ and has bounded incremental
    gain.
    Suppose 
    \begin{enumerate}[label=(\roman*)]
        \item  $H_1, H_2: L_2 \to L_2$  have finite gain with zero offset 
    and are causal;
        \item $[H_1, \tau H_2]$ 
    is well-posed and causal for all $\tau \in (0, 1]$;
        \item  $H_1$ and $H_2$ have strictly separated SRGs.
    \end{enumerate}
    Then $[H_1, H_2]$ maps
    $L_2$ to $L_2$ and has finite incremental gain.
\end{corollary}

Although we don't pursue it
here, (non-incremental) finite gain can be verified
using separation of (non-incremental) Scaled Graphs, as defined in
\autocite{Chaffey2023}, using a straightforward adaptation of
Theorem~\ref{thm:homotopy_non_incremental}.

\printbibliography

@article{Chaffey2023,
  title = {Graphical {{Nonlinear System Analysis}}},
  author = {Chaffey, Thomas and Forni, Fulvio and Sepulchre, Rodolphe},
  date = {2023},
  journaltitle = {IEEE Transactions on Automatic Control},
  pages = {1--16},
  issn = {1558-2523},
  doi = {10.1109/TAC.2023.3234016},
  eventtitle = {{{IEEE Transactions}} on {{Automatic Control}}}
}

@inproceedings{Chen2024,
  title = {On the {{Scaled Relative Graph Separation}} for {{Feedback Incremental Stability}}},
  booktitle = {Benelux {{Meeting}} 2024},
  author = {Chen, Chao and Sepulchre, Rodolphe},
  date = {2024},
  langid = {english}
}

@book{Desoer1975,
  title = {Feedback {{Systems}}: {{Input}}–{{Output Properties}}},
  author = {Desoer, Charles A. and Vidyasagar, Mathukumalli},
  date = {1975},
  publisher = {Elsevier},
  doi = {10.1016/b978-0-12-212050-3.x5001-4}
}

@incollection{Freeman2022,
  title = {On the {{Role}} of {{Well-Posedness}} in {{Homotopy Methods}} for the {{Stability Analysis}} of {{Nonlinear Feedback Systems}}},
  booktitle = {Trends in {{Nonlinear}} and {{Adaptive Control}}: {{A Tribute}} to {{Laurent Praly}} for His 65th {{Birthday}}},
  author = {Freeman, Randy A.},
  editor = {Jiang, Zhong-Ping and Prieur, Christophe and Astolfi, Alessandro},
  date = {2022},
  series = {Lecture {{Notes}} in {{Control}} and {{Information Sciences}}},
  pages = {43--82},
  publisher = {Springer International Publishing},
  location = {Cham},
  doi = {10.1007/978-3-030-74628-5_3},
  url = {https://doi.org/10.1007/978-3-030-74628-5_3},
  urldate = {2024-01-15},
  isbn = {978-3-030-74628-5},
  langid = {english}
}

@article{Fromion1996,
  title = {A Link between Input-Output Stability and {{Lyapunov}} Stability},
  author = {Fromion, V. and Monaco, S. and Normand-Cyrot, D.},
  date = {1996-04-15},
  journaltitle = {Systems \& Control Letters},
  shortjournal = {Systems \& Control Letters},
  volume = {27},
  number = {4},
  pages = {243--248},
  issn = {0167-6911},
  doi = {10.1016/0167-6911(95)00046-1},
  url = {https://www.sciencedirect.com/science/article/pii/0167691195000461},
  urldate = {2024-01-15}
}

@article{Georgiou1997,
  title = {Robustness Analysis of Nonlinear Feedback Systems: An Input-Output Approach},
  shorttitle = {Robustness Analysis of Nonlinear Feedback Systems},
  author = {Georgiou, Tryphon and Smith, Malcolm},
  date = {1997},
  journaltitle = {IEEE Transactions on Automatic Control},
  volume = {42},
  number = {9},
  pages = {1200--1221},
  doi = {10.1109/9.623082},
  eventtitle = {{{IEEE Transactions}} on {{Automatic Control}}}
}

@inproceedings{Gronqvist2022,
  title = {Integral {{Quadratic Constraints}} for {{Neural Networks}}},
  booktitle = {2022 {{European Control Conference}} ({{ECC}})},
  author = {Gronqvist, Johan and Rantzer, Anders},
  date = {2022-07-12},
  pages = {1864--1869},
  publisher = {IEEE},
  location = {London, United Kingdom},
  doi = {10.23919/ECC55457.2022.9838065},
  url = {https://ieeexplore.ieee.org/document/9838065/},
  urldate = {2024-01-14},
  eventtitle = {2022 {{European Control Conference}} ({{ECC}})},
  isbn = {978-3-907144-07-7},
  langid = {english}
}

@inproceedings{Jonsson2001,
  ids = {Jonsson2001a},
  title = {A Semi-Infinite Optimization Problem in Harmonic Analysis of Uncertain Systems},
  booktitle = {Proceedings of the 2001 {{American Control Conference}}},
  author = {Jonsson, U. T. and {Chung-Yao Kao} and Megretski, A.},
  date = {2001},
  volume = {4},
  pages = {3029-3034 vol.4},
  issn = {0743-1619},
  doi = {10.1109/ACC.2001.946379},
  eventtitle = {Proceedings of the 2001 {{American Control Conference}}}
}

@article{Megretski1997,
  title = {System Analysis via Integral Quadratic Constraints},
  author = {Megretski, A. and Rantzer, A.},
  date = {1997},
  journaltitle = {IEEE Transactions on Automatic Control},
  volume = {42},
  number = {6},
  pages = {819--830},
  doi = {10.1109/9.587335}
}

@online{Pates2021,
  title = {The {{Scaled Relative Graph}} of a {{Linear Operator}}},
  author = {Pates, Richard},
  date = {2021},
  eprint = {2106.05650},
  eprinttype = {arXiv},
  doi = {10.48550/arXiv.2106.05650},
  pubstate = {prepublished}
}

@report{Rantzer1997,
  title = {System {{Analysis}} via {{Integral Quadratic Constraints Part II}}},
  author = {Rantzer, Anders and Megretski, Alexander},
  date = {1997},
  number = {ISRN LUTFD 2 / TFRT– 7559– SE},
  location = {Lund Institue of Technology}
}

@article{Ryu2021,
  ids = {Ryu2021b},
  title = {Scaled Relative Graphs: Nonexpansive Operators via {{2D Euclidean}} Geometry},
  shorttitle = {Scaled Relative Graphs},
  author = {Ryu, Ernest K. and Hannah, Robert and Yin, Wotao},
  date = {2021},
  journaltitle = {Mathematical Programming},
  shortjournal = {Math. Program.},
  doi = {10.1007/s10107-021-01639-w},
  langid = {english}
}

@unpublished{Scherer2015,
  title = {Linear {{Matrix Inequalities}} in {{Control}}},
  author = {Scherer, Carsten and Weiland, Siep},
  date = {2015}
}

@article{vanWaarde2023,
  title = {Kernel-{{Based Models}} for {{System Analysis}}},
  author = {family=Waarde, given=Henk J., prefix=van, useprefix=true and Sepulchre, Rodolphe},
  date = {2023-09},
  journaltitle = {IEEE Transactions on Automatic Control},
  shortjournal = {IEEE Trans. Automat. Contr.},
  volume = {68},
  number = {9},
  pages = {5317--5332},
  issn = {0018-9286, 1558-2523, 2334-3303},
  doi = {10.1109/TAC.2022.3218944},
  url = {https://ieeexplore.ieee.org/document/9935290/},
  urldate = {2024-01-15},
  langid = {english}
}

@inproceedings{Wang2019,
  title = {Robust Contraction Analysis of Nonlinear Systems via Differential {{IQC}}},
  booktitle = {2019 {{IEEE}} 58th {{Conference}} on {{Decision}} and {{Control}} ({{CDC}})},
  author = {Wang, Ruigang and Manchester, Ian R.},
  date = {2019-12},
  pages = {6766--6771},
  publisher = {IEEE},
  location = {Nice, France},
  doi = {10.1109/CDC40024.2019.9029867},
  url = {https://ieeexplore.ieee.org/document/9029867/},
  urldate = {2021-08-06},
  eventtitle = {2019 {{IEEE}} 58th {{Conference}} on {{Decision}} and {{Control}} ({{CDC}})},
  isbn = {978-1-72811-398-2}
}

@article{Zames1966a,
  title = {On the Input-Output Stability of Time-Varying Nonlinear Feedback Systems, Part One: Conditions Derived Using Concepts of Loop Gain, Conicity, and Positivity},
  author = {Zames, G.},
  date = {1966},
  journaltitle = {IEEE Transactions on Automatic Control},
  volume = {11},
  number = {2},
  pages = {228--238},
  doi = {10.1109/tac.1966.1098316},
  eventtitle = {{{IEEE Transactions}} on {{Automatic Control}}}
}
\end{document}